\documentclass[11pt,a4paper]{amsart}
\usepackage{amssymb,amsthm,amsfonts,amsmath,graphicx,color}
\usepackage[numbers,sort&compress]{natbib}
\usepackage[lmargin=34mm,rmargin=34mm,tmargin=34mm,bmargin=34mm]{geometry}
\renewcommand{\baselinestretch}{1.22}
\usepackage[colorlinks=true,citecolor=black,linkcolor=black]{hyperref}
\newcommand{\doi}[1]{\href{http://dx.doi.org/#1}{\texttt{doi:#1}}}
\newcommand{\urlprefix}{}
\theoremstyle{plain} 
\newtheorem{theorem}{Theorem}
\newcommand{\thmlabel}[1]{\label{thm:#1}}
\newcommand{\thmref}[1]{Theorem~\ref{thm:#1}}
\newcommand{\CartProd}{\ensuremath{\square}}
\newcommand{\CP}[2]{\ensuremath{#1{\,\CartProd\,}#2}}
\newcommand{\tw}[1]{\ensuremath{\textup{\textsf{tw}}(#1)}}
\newcommand{\bw}[1]{\ensuremath{\textup{\textsf{bw}}(#1)}}
\newcommand{\B}{\ensuremath{\mathcal{B}}}
\begin{document}
\title[Treewidth of Cartesian Products of Highly Connected
Graphs]{Treewidth of Cartesian Products \\ of Highly Connected Graphs}
\author{David~R.~Wood} \thanks{Department of Mathematics and
  Statistics, The University of Melbourne, Australia
  (\texttt{woodd@unimelb.edu.au}).  Supported by QEII Research
  Fellowship from the Australian Research Council. } \date{\today}
\begin{abstract}
  The following theorem is proved: For all $k$-connected graphs $G$
  and $H$ each with at least $n$ vertices, the treewidth of the
  cartesian product of $G$ and $H$ is at least $k(n -2k+2)-1$. For
  $n\gg k$ this lower bound is asymptotically tight for particular
  graphs $G$ and $H$. This theorem generalises a well known result
  about the treewidth of planar grid graphs.
\end{abstract}

\maketitle

Treewidth is a graph parameter of fundamental importance in graph
minor theory, with numerous applications in algorithmic theory and
practical computing. The planar grid graph is a key example for
treewidth, in that the $n\times n$ planar grid has treewidth $n$, and
every graph with sufficiently large treewidth contains the $n\times n$
planar grid as a minor.

Motivated by the fact that the planar grid can be defined to be the
cartesian product of two paths, in this note we consider the treewidth
of cartesian products of general graphs. Our main result is a lower
bound on the treewidth of the cartesian product of two highly
connected graphs; see \citep{Djelloul-TCS09, Chvat-DM75, KA-DM02,
  Harper-TCS03, Harper-JCT66, BBHS-TCS08, CK-DM06, Wood-ProductMinor,
  WYY10,OS11, FitzGerald-MC74} for related results. Before stating the
theorem, we introduce the necessary definitions.

The \emph{cartesian product} of graphs $G$ and $H$, denoted by
\CP{G}{H}, is the graph with vertex set $V(\CP{G}{H}):=V(G)\times
V(H)$, where $(v,x)(w,y)$ is an edge of \CP{G}{H} if and only if
$vw\in E(G)$ and $x=y$, or $v=w$ and $xy\in E(H)$.  For each vertex
$v\in V(G)$ the subgraph of \CP{G}{H} induced by $\{(v,w):w\in V(H)\}$
is isomorphic to $H$; we call it the \emph{$v$-copy} of $H$, denoted
by $H_v$. Similarly, for each vertex $w\in V(H)$ the subgraph of
\CP{G}{H} induced by $\{(v,w):v\in V(G)\}$ is isomorphic to $G$; we
call it the \emph{$w$-copy} of $G$, denoted by $G_w$.

A \emph{tree decomposition} of a graph $G$ consists of a tree $T$ and
a set $\{T_x\subseteq V(G):x\in V(T)\}$ of `bags' of vertices of $G$
indexed by $T$, such that
\begin{itemize}
\item for each edge $vw\in E(G)$, some bag $T_x$ contains both $v$ and
  $w$, and
\item for each vertex $v\in V(G)$, the set $\{x\in V(T):v\in T_x\}$
  induces a non-empty (connected) subtree of $T$.
\end{itemize}
The \emph{width} of the tree decomposition is $\max\{|T_x|:x\in
V(T)\}-1$. The \emph{treewidth} of $G$, denoted by \tw{G}, is the
minimum width of a tree decomposition of $G$. For example, $G$ has
treewidth $1$ if and only if $G$ is a forest.

Let $G$ be a graph. Two subgraphs $X$ and $Y$ of $G$ \emph{touch} if
$X\cap Y\neq\emptyset$ or there is an edge of $G$ between $X$ and
$Y$. A \emph{bramble} in $G$ is a set of pairwise touching connected
subgraphs. A set $S$ of vertices in $G$ is a \emph{hitting set} of a
bramble \B\ if $S$ intersects every element of \B. The \emph{order} of
\B\ is the minimum size of a hitting set. The canonical example of a
bramble of order $\ell$ is the set of crosses (union of a row and
column) in the $\ell\times\ell$ grid. The following `Treewidth Duality
Theorem' shows the intimate relationship between treewidth and
brambles; see \cite{BD-CPC02} for an alternative proof.

\begin{theorem}[\cite{SeymourThomas-JCTB93}]
  \thmlabel{TreewidthBramble} A graph $G$ has treewidth at least
  $\ell$ if and only if $G$ contains a bramble of order at least
  $\ell+1$.
\end{theorem}

This paper proves the following general lower bound on the treewidth
of cartesian products of highly connected graphs.

\begin{theorem}
  \thmlabel{ConnectivityTreewidth} For all $k$-connected graphs $G$
  and $H$ each with at least $n$ vertices,
$$\tw{\CP{G}{H}}\geq k(n -2k+2)-1\enspace.$$
\end{theorem}

\begin{proof}
  To prove this theorem, we construct a bramble \B\ in \CP{G}{H} and
  then apply \thmref{TreewidthBramble}.  If $n\leq 2k-2$ then the
  claim is vacuously true. Now assume that $n\geq 2k-1$.

  Let $\B$ be the set of all subgraphs $X$ of \CP{G}{H} formed in the
  following way. Let $S$ be a set of $2k-1$ vertices in $G$.  Let $T$
  be a set of $2k-1$ vertices in $H$.  Initialise $X$ to be the union
  of $\cup\{H_v:v\in S\}$ and $\cup\{G_w:w\in T\}$.  Now delete
  vertices from $X$ such that at most $k-1$ vertices are deleted from
  $H_v$ for each $v\in S$, and at most $k-1$ vertices are deleted from
  $G_w$ for each $w\in T$. We claim that \B\ is a bramble of
  \CP{G}{H}.

  First we prove that each $X\in \B$ is connected. Say $X$ is defined
  with respect to $S\subseteq V(G)$ and $T\subseteq V(H)$.  First note
  that for each $v\in S$ and $w\in T$, since $H_v$ and $G_w$ are
  $k$-connected, $H_v\cap X$ and $G_w\cap X$ are connected.  Let $Q$
  be the bipartite graph with $V(Q)=S\cup T$, where for all $v\in S$
  and $w\in T$, the edge $vw$ is in $Q$ whenever the vertex $(v,w)$ is
  in $X$ (i.e., it was not deleted).  The degree in $Q$ of each vertex
  $v\in S$ is at least $(2k-1)-(k-1)=k$ since at most $k-1$ vertices
  were deleted from $H_v$. Similarly, each vertex in $T$ has degree at
  least $k$ in $Q$. So $Q$ has $2k-1$ vertices in each colour class,
  and minimum degree $k$. If $Q$ is disconnected then some component
  $H$ of $Q$ contains at most $k-1$ vertices in $S$, implying that the
  vertices in $H\cap T$ have degree at most $k-1$. Thus $Q$ is
  connected.  Now consider two vertices $(v_1,w_1)$ and $(v_2,w_2)$ in
  $X$.  Thus $v_1w_1$ and $v_2w_2$ are edges of $Q$.  Since $Q$ is
  connected, there is a path $P$ in $Q$ between one endpoint of
  $v_1w_1$ and one endpoint of $v_2w_2$. For each 2-edge path $v w v'$
  of $P$, since $G_{w}\cap X$ is connected, there is a path in $X$
  between the vertices $(v,w)$ and $(v' ,w)$.  Similarly, for each
  2-edge path $w v w'$ of $P$, there is a path in $X$ between the
  vertices $(v,w)$ and $(v,w')$.  The union of these paths is a walk
  between $(v_1,w_1)$ and $(v_2,w_2)$ in $X$. Therefore $X$ is
  connected, as claimed.

  Now we prove that $X$ and $X'$ touch for all $X,X'\in\B$. Say $X$ is
  defined with respect to $S$ and $T$, and $X'$ is defined with
  respect to $S'$ and $T'$.  At most $(2k-1)(k-1)$ vertices in
  $S\times T'$ were deleted in the construction of $X$, and at most
  $(2k-1)(k-1)$ vertices in $S\times T'$ were deleted in the
  construction of $X'$. Since $|S\times T'|=(2k-1)^2>2(2k-1)(k-1)$,
  some vertex $(v,w)\in S\times T'$ was deleted in neither the
  construction of $X$ nor the construction of $X'$. Hence $(v,w)$ is
  in both $H_v\cap X$ and $G_w\cap X$. Thus $X$ and $X'$ have a common
  vertex.

  Therefore \B\ is a bramble. Let $J$ be a hitting set of \B.  We
  claim that $|J|\geq k(n-2k+2)$.  Let $S_0:=\{v\in V(G):|V(H_v)\cap
  J|\leq k-1\}$ and $T_0:=\{w\in V(H):|V(G_w)\cap J|\leq k-1\}$.  If
  $|S_0|\leq 2k-2$ then at least $n-(2k-2)$ pairwise-disjoint copies
  of $H$ contain at least $k$ vertices in $J$, implying $|J|\geq
  k(n-2k+2)$, as claimed. Otherwise, $|S_0|\geq 2k-1$. Similarly,
  $|T_0|\geq 2k-1$. Let $S\subseteq S_0$ and $T\subseteq T_0$ such
  that $|S|=|T|=2k-1$.  Let $X$ be the union of $\cup\{H_v- J:v\in
  S\}$ and $\cup\{G_w- J:w\in T\}$.  Thus $X\in\B$ (since $|V(H_v\cap
  J)|\leq k-1$ and $|V(G_w\cap J)|\leq k-1$ for each $v\in S$ and
  $w\in T$). However, $X\cap J=\emptyset$. Thus $J$ is not a hitting
  set for \B.  Hence the order of \B\ is at least $k(n-2k+2)$.  The
  result follows from \thmref{TreewidthBramble}.
\end{proof}

We now show that the bound in \thmref{ConnectivityTreewidth} is tight
(ignoring lower order terms and assuming $n\gg k$).  The
\emph{bandwidth} of a graph $G$, denoted by $\bw{G}$, is the minimum,
taken over all bijections $\phi:V(G)\rightarrow\{1,2,\dots,|V(G)|\}$,
of the maximum, taken over all edges $vw\in E(G)$, of
$|\phi(v)-|\phi(w)|$.  It is well known that $\tw{G}\leq\bw{G}$; see
\citep{Bodlaender-TCS98}.  Let $P_n^k$ be the $k$-th power of a path,
which has vertex set $\{1,2,\dots,n\}$, where $ij$ is an edge if and
only if $|i-j|\leq k$. Clearly $P_n^k$ is $k$-connected.  Let $\phi$
be the vertex ordering of \CP{P_n^k}{P_n^k} defined by
$\phi((x,y))=(x-1)n+y$. Each edge $(x,y)(x,y')$ has width
$|y'-y|\leq k$, and each edge $(x,y)(x',y)$ has width
$|(x'-x)n|\leq kn$. Hence $\tw{\CP{P_n^k}{P_n^k}}\leq \bw{\CP{P_n^k}{P_n^k}}\leq kn$.
(This upper bound can be slightly improved by ordering the vertices
with respect to the function $x(n+1)+yn$.)\ 

In fact, there is a much broader class of graphs that provide an upper
bound only slightly weaker than $kn$. Let $G$ and $H$ be  chordal graphs with $n$ vertices and
connectivity $k$. It is well known that $G$ and $H$ have
clique-number $k+1$ and treewidth $k$. A tree decomposition of $G$ with
width $k$ can be easily turned into a tree decomposition of
$\CP{G}{H}$ with width $(k+1)n-1$; see
\citep{Wood-ProductMinor,Djelloul-TCS09}. Thus
$\tw{\CP{G}{H}}\leq(k+1)n-1$. 

We expect that the dependence on $k$ in \thmref{ConnectivityTreewidth}
can be slightly improved (although it is not obvious how to do
so). For example, \thmref{ConnectivityTreewidth} with $k=1$ implies
that the $n\times n$ grid has treewidth at least $n-1$, whereas it
actually has treewidth $n$; see \citep{BGK-Algo08} for a proof.
Another interesting example is the toroidal grid graph \CP{C_n}{C_n}.
By \thmref{ConnectivityTreewidth} with $k=2$ and since $C_n$ is a
subgraph of $P_n^2$ ,
$$2n-5\leq \tw{\CP{C_n}{C_n}}\leq \tw{\CP{P_n^2}{P_n^2}}\leq
2n\enspace.$$ 


\def\cprime{$'$} \def\soft#1{\leavevmode\setbox0=\hbox{h}\dimen7=\ht0\advance
  \dimen7 by-1ex\relax\if t#1\relax\rlap{\raise.6\dimen7
  \hbox{\kern.3ex\char'47}}#1\relax\else\if T#1\relax
  \rlap{\raise.5\dimen7\hbox{\kern1.3ex\char'47}}#1\relax \else\if
  d#1\relax\rlap{\raise.5\dimen7\hbox{\kern.9ex \char'47}}#1\relax\else\if
  D#1\relax\rlap{\raise.5\dimen7 \hbox{\kern1.4ex\char'47}}#1\relax\else\if
  l#1\relax \rlap{\raise.5\dimen7\hbox{\kern.4ex\char'47}}#1\relax \else\if
  L#1\relax\rlap{\raise.5\dimen7\hbox{\kern.7ex
  \char'47}}#1\relax\else\message{accent \string\soft \space #1 not
  defined!}#1\relax\fi\fi\fi\fi\fi\fi}

\end{document}